\RequirePackage{fix-cm}
\documentclass[smallcondensed]{svjour3}

\smartqed
\usepackage{hyperref}
\usepackage{amsmath}
\usepackage{amsfonts,amssymb}

\makeatletter

\newenvironment{indentation}[2]%
{%
  \par
  \setlength{\leftmargin}{#1}%
  \setlength{\rightmargin}{#2}%
  \advance\linewidth -\leftmargin
  \advance\linewidth -\rightmargin
  \advance\@totalleftmargin\leftmargin
  \@setpar{{\@@par}}%
  \parshape 1 \@totalleftmargin \linewidth
  \ignorespaces
}
{\par\vspace*{.3em}}

\makeatother

\newenvironment{lproof}{{\emph{Proof.} }}{\hfill $\triangle$ \\}

\newcommand{\F}{\mathbb{F}}
\newcommand{\C}{{\bf C}}
\newcommand{\Z}{{\bf Z}}

\newcommand{\m}[1]{\mathfrak{#1} }
\newcommand{\Ker}{\operatorname{Ker} }

\allowdisplaybreaks

\begin{document}

\title{Kernels of linear maps}
\subtitle{A generalization of Duistermaat and Van der Kallens theorem}
\author{Arno van den Essen \and Jan Schoone}
\institute{
		Arno van den Essen \at Radboud University, Nijmegen
		\and
		Jan Schoone \at Digital Security, Radboud University, Nijmegen\texorpdfstring{\\\email{jan.schoone@ru.nl}}{}}
\date{Received: date / Accepted: date}

\maketitle

\begin{abstract}
The theorem of Duistermaat and Van der Kallen from 1998 proved the first case of the Mathieu conjecture.
Using the theory of Mathieu-Zhao spaces, we can reformulate this theorem as $\Ker L$ is a Mathieu-Zhao space where $L$ is the linear map 
    \begin{align*}
        L \colon \C[X_1,\ldots,X_n,X_1,\ldots,X_n^{-1}] \to \C,\ f \mapsto f_0.
    \end{align*}
In this paper, we generalize this result (for $n=1$) to all non-trivial linear maps $L\colon \C[X,X^{-1}] \to \C$ such that $\{ X^n \mid |n| \geq N\} \subset \Ker L$ for some $N\geq 1$.
\end{abstract}

\section{Introduction}

In 1939, Ott-Heinrich Keller \cite{keller} formulated what is now known as the Jacobian conjecture:

\begin{conjecture}[Jacobian conjecture] Let $k$ be a field of characteristic zero and $F \in k[X_1,\ldots,X_n]^n$ a polynomial mapping such that its Jacobian determinant $J_F$ is a non-zero constant.
Then $F$ has an inverse function $G \in k[X_1,\ldots,X_n]^n$.
\end{conjecture}

Since then, this conjecture remained unsolved (apart from the case $n=1$).
For a historical treaty and a collection of results around the Jacobian conjecture, one can read \cite{Ess0} and \cite{Ess2}.
In 1997, Olivier Mathieu showed in \cite{Mat} that the Jacobian conjecture follows from what is now known as the Mathieu conjecture:

\begin{conjecture}[Mathieu conjecture] Let $G$ be a compact connected real Lie group with Haar measure $\sigma$.
Let $f$ be a complex valued $G$-finite function on $G$ such that $\int_G f^m d\sigma = 0$ for all $m\geq 1$.
Then for every $G$-finite function $g$ on $G$, also $\int_G gf^m d\sigma = 0$ for all large $m$.
\end{conjecture}

The Mathieu conjecture has been proved for the case of connected abelian groups by Duistermaat and Van der Kallen in 1998 \cite{DuiKal}.

From the formulation of the Mathieu conjecture, Wenhua Zhao formulated the notion of Mathieu-Zhao spaces that are a generalization of ideals. 
For history and details, see \cite{Zhao1} and \cite{Ess2}.

Together with this formulation, we manage to prove a generalization of the theorem of Duistermaat and Van der Kallen in dimension one.

\section{Notations and conventions} 

We write $\Z$ for the ring of integers, $\C$ for the field of complex numbers, and $\F_q$ for a finite field of $q$ elements.
Furthermore, we always assume all rings to be commutative and to have an identity.
Similarly, we assume all algebras to be unital and associative. 

\section{Mathieu-Zhao spaces and partial fractions decomposition}

Because of the shape of the statement in the Mathieu conjecture \cite{Mat}, we can define the following generalisation of ideals in associative algebras, which we call Mathieu-Zhao spaces.

\begin{definition} 
Let $R$ be any ring and $A$ an $R$-algebra. 
For any subset $M$ of $A$ we define the \emph{radical of $M$} by
    \begin{align*} r(M) := \{ a \in A \mid \forall m \gg 0: a^m \in M\}\end{align*}
with $m \gg 0$ meaning all $m \geq N$ for some $N \geq 1$.
\end{definition}

\begin{remark} 
For an ideal $I$ of $R$, this coincides with the usual definition of the radical of $I$, being:
    \begin{align*} \sqrt{I} := \{ a \in R \mid \exists n : a^n \in I\}.\end{align*}
It is well-known that the radical of an ideal is an ideal itself, and hence all $a^{n+i}$ are elements of $I$ for $i\geq 0$. 
\end{remark}

\begin{definition} 
Let $R$ be any ring and $A$ an $R$-algebra. 
For any subset $M$ of $A$ we define the \emph{(left) strong radical of $M$} by
    \begin{align*} sr(M) := \{ a\in A \mid \forall b\in A\ \forall m \gg 0: ba^m \in M \}.\end{align*}
\end{definition}

\begin{remark} 
Note that if $a\in sr(M)$, then $1\cdot a^m \in M$ for all $ m \gg 0$, hence $a\in r(M)$. 
So we always have $sr(M)\subset r(M)$. 
\end{remark}

\begin{definition} 
Let $R$ be any ring and $A$ an $R$-algebra. 
We say that an $R$-submodule $M$ of $A$ is a Mathieu-Zhao space of $A$ (short: MZ-space) if and only if $sr(M)=r(M)$. 
\end{definition}

A little less is needed in practice. 
Define $r'(M) = \{ a \in A \mid \forall m \geq 1: a^m \in M\}$.  If $r'(M) \subset sr(M)$, then $M$ is an MZ-space of $A$.

\begin{proposition}[\cite{Ess2}] \label{cor224} 
Let $R$ be a ring and $A$ an $R$-algebra. 
Let $M$ be an additive subset of $A$ and $I\subset A$ an ideal of $A$ such that $I\subset M$. 
Then $M$ is an MZ-space of $A$ if and only if $M/I$ is an MZ-space of $A/I$.
\end{proposition}

\begin{example} 
Let $R$ be a ring and $A$ an $R$-algebra. 
Let $I$ be an ideal of $A$.
Then $I$ is an MZ-space.
\end{example}
\begin{lproof}
We need to show that $r'(M) \subset sr(M)$.
Thus, let $a\in r'(I)$ be arbitrary.
Then $a\in I$. Since $I$ is an ideal, we have for any $b\in A$ that $ba^na\in I$ for all $n\geq 0$. 
Therefore $a\in sr(I)$.
\end{lproof}

\subsection{Zhao's Idempotency Theorem}

One of the most important theorems regarding MZ-spaces is Zhao's Idempotency Theorem:

\begin{theorem}[Zhao's Idempotency Theorem \cite{Zhao1}] \label{zhao} 
Let $k$ be a field and $A$ a $k$-algebra. 
Let $M$ be a $k$-linear subspace of $A$ such that all elements of $r(M)$ are algebraic over $k$. 
Then $M$ is an MZ-space of $A$ if and only if $Ae \subset M$ for all idempotents $e$ which belong to $M$.
\end{theorem}

In particular, if $M$ is an MZ-space of $A$ and $1 \in M$, then $M=A$.

\begin{corollary} \label{suchashame} 
Let $A$ be a $k$-algebra and suppose that all elements of $r(M)$ are algebraic over $k$. 
Then $M$ is an MZ-space of $A$ if and only if $r(M)$ is an ideal of $A$.
\end{corollary}

Not all MZ-spaces are ideals, as the following example shows:

\begin{example} 
The finite field $\F_4$ has the following four MZ-spaces:
    \begin{align*}
        \F_4 \qquad \{0\} \qquad \{0,x\} \qquad \{0,x+1\}.
    \end{align*}
\end{example}

\begin{lproof} 
    The sets $\{0\}$, $\{0,x\}$ and $\{0,x+1\}$ are the only $\Z$-linear subspaces of $\F_4$ that do not contain $1$.
    Since $\F_4^*$ is cyclic, we find that $r'(M) = \{0\}$ when $M$ is a $\Z$-linear subspace that does not contain $1$.
    It is clear that $r'(M) \subset sr(M)$, thus $M$ is an MZ-space.
\end{lproof}

\subsection{Partial fractions decomposition}

We have the partial fractions decomposition of rational maps:

\begin{theorem} \label{pfd} \label{gold} \emph{(Partial Fractions Decomposition)} 
Let $k$ be a field, $a_1,\ldots,a_n \in k$ disinct and $\alpha\in k^*$. 
Define $U(X) = \alpha(X-a_1)\cdots(X-a_n)$ and let $V(X)\in k[X]$ with $\deg_X V(X)<n$ be arbitrary. 
Then
    \begin{align*} \frac{V(X)}{U(X)} = \sum_{i=1}^n \frac{A_i}{X-a_i}\end{align*}
with $A_i = \frac{V(a_i)}{U'(a_i)}$.
\end{theorem}
\begin{proof} 
Note first that $U'(a_i) \neq 0$ for all $i=1,\ldots,n$. 
This is since all $a_i$ are distinct and we know that if both $U(a_i)=0$ and $U'(a_i)=0$, then $a_i$ is a double root of $U$. 
Since this is not the case, and we know that $U(a_i)=0$, we must have $U'(a_i)\neq 0$ for all $i=1,\ldots,n$.

Then consider the set $\mathcal{V}_{U(X)} := \left\{ \frac{V(X)}{U(X)} \mid \deg V(X) <n\right\}$ as a subset of $k(X)$, the field of rational functions. 
Since $\deg V(X)<n$, we see that $\mathcal{V}_{U(X)}$ is spanned by the set $\left\{\frac{1}{U(X)}, \frac{X}{U(X)},\ldots, \frac{X^{n-1}}{U(X)}\right\}$.
It is clearly a $k$-linear subspace of $k(X)$.
Indeed, the elements $\frac{1}{U(X)}, \frac{X}{U(X)},\ldots, \frac{X^{n-1}}{U(X)}$ are $k$-linearly independent as well. 
Hence $\dim_k \mathcal{V}_{U(X)} = n$.

Now consider the elements $\frac1{X-a_1},\ldots, \frac1{X-a_n}$. 
One easily verifies that they belong to $\mathcal{V}_{U(X)}$. 
We show that they are linearly independent, and hence form a basis. 
Therefore, suppose that there exist $\lambda_1,\ldots,\lambda_n\in k$ such that
    \begin{align*} \sum_{i=1}^n \frac{\lambda_i}{X-a_i} = 0.\end{align*}
Then also
    \begin{align*} \lambda_1 + \sum_{i=2}^n \frac{\lambda_i(X-a_1)}{X-a_i} = 0.\end{align*}
By substituting $a_1$ for $X$, we find $\lambda_1=0$. 

Then $\sum_{i=2}^n \frac{\lambda_i}{X-a_i} = 0$ and we repeat this procedure. 
We find that indeed $\frac1{X-a_1},\ldots,\frac1{X-a_n}$ are linearly independent.

Since now these elements constitute a basis for $\mathcal{V}_{U(X)}$, given any $V(X)$ with $\deg V(X)<n$, we can write
    \begin{align*} \frac{V(X)}{U(X)} = \sum_{i=1}^n \frac{A_i}{X-a_i}\end{align*}
for some $A_1,\ldots,A_n\in k$. 
We now determine those $A_i$. 
By multiplying both sides with $(X-a_j)$ we get
    \begin{align*} \frac{V(X)}{U(X)} \cdot (X-a_j) = A_j + \sum_{i\neq j} \frac{A_i(X-a_j)}{X-a_i}.\end{align*}
Note that the $X-a_j$ cancels with the $X-a_j$ that occurs in $U(X)$. 
So we can substitute $a_j$ for $X$ and obtain
    \begin{align*} \frac{V(a_j)}{U'(a_j)} = A_j,\end{align*}
as required. 
Note that we used that $U'(a_j) = \alpha \prod_{i\neq j} (a_j-a_i)$. \qed
\end{proof}

\section{The Theorem of Duistermaat and Van der Kallen}

In this section we state the Theorem of Duistermaat and Van der Kallen in one dimension. 
We shall prove a more general theorem later on, adapting a technique by Monsky. 
Before we do that, we need to discuss some theorems, all of which are self-contained and needed for our proof. 
At the end, we shall state the Theorem of Duistermaat and Van der Kallen. 
The first necessary theorem is the Newton-Puiseux Theorem. 
(See: \cite{New}, \cite{Pui}, \cite{Pui2})

\begin{theorem}[Newton-Puiseux Theorem \cite{New}, \cite{Pui}, \cite{Pui2}] \label{newtonpuiseux} 
Let $k$ be an algebraically closed field of characteristic zero and $f(X) = \sum_{i=0}^n a_i(t) X^i \in k((t))[X]$ with $n:=\deg_X f(X)\geq 1$. 
Then there exists some $p\geq 1$ such that $f(X)$ splits completely in linear factors over $k((t^{1/p}))$.
\end{theorem}

Next, we need a special ring for our proof:

\begin{definition} \label{spandauballet} 
Let $k$ be a field and $v$ a valuation on $k$ such that $k$ is complete with respect to $v$. 
Define $k[[X,X^{-1}]]$ to be the ring of formal series $\sum_{n=-\infty}^{\infty} c_nX^n$ such that $\lim_{|n|\to\infty} c_n = 0$. 
\end{definition}

\begin{proposition} \label{true} 
Let $k = \C((z^{1/p}))$ and $v$ the valuation of $k$ such that $v(z^{1/p}) = 1/p$. 
Then $\C[X,X^{-1}][[z]]$ is a subring of $k[[X,X^{-1}]]$.
\end{proposition}
\begin{proof} 
Note that $\C((z^{1/p}))$ is complete with respect to $v$, one could write $z^{1/p}=t$, then $k= \C((t))$ and $v(t)=1/p$. 
Hence $k[[X,X^{-1}]]$ is defined.

Let $f(z) = a_0(X) + a_1(X)z + a_2(X)z^2 +\ldots$ be an arbitrary element of $\C[X,X^{-1}][[z]]$ and let $\varepsilon>0$ be arbitrary. 
Write this element $f(z) = \sum_{n=-\infty}^{\infty} c_n X^n$. 

We want to determine $N$ such that $|c_n| < \varepsilon$ for all $|n|\geq N$. 
First, observe that $v(z^i) = iv(z) = i/p$.
So $|z^i| = 2^{-v(z^i)} = 2^{-i/p} < \varepsilon$ if $i > p\log_2(\frac1\varepsilon)$.

Write $N(\varepsilon):= \lceil p\log_2(\frac1\varepsilon) \rceil$. 
Then determine 
    \begin{align*}
        N:= \max_{0\leq i \leq N(\varepsilon)} \{ \deg a_i(X), \deg a_i(X^{-1})\}.
    \end{align*} 
Then, if $|n|>N$, all $z^i$ contributing to $c_n$ come from $a_i(X)$ with $i>N(\varepsilon)$.
Since $|z^i|<\varepsilon$ if $i>N(\varepsilon)$, it follows that $|c_n|<\varepsilon$.
Hence $|c_n|< \varepsilon$ if $|n|>N$.

Hence indeed, $f(z)\in k[[X,X^{-1}]]$, as required. \qed
\end{proof}

The theorem of Duistermaat and Van der Kallen is as follows:

\begin{theorem}[Duistermaat - Van der Kallen \cite{DuiKal}] 
Let $X_1,\ldots,X_n$ be $n$ commutative variables and let $M$ be the subspace of the Laurent polynomial algebra $\C[X_1,\ldots,X_n,X_1^{-1},\ldots,X_n^{-1}]$ consisting of those Laurent polynomials with no constant term:
    \begin{align*} M = \{ f \in \C[X_1,\ldots,X_n,X_1^{-1},\ldots,X_n^{-1}] \mid f_0 = 0 \}. \end{align*}
Then $M$ is an MZ-space of $\C[X_1,\ldots,X_n,X_1^{-1},\ldots,X_n^{-1}]$.
\end{theorem}

We omit the proof, and prove a more general theorem, for dimension $n=1$, in the next section, see Theorem \ref{moregeneral}. 

\section{Kernels of linear maps: A generalisation of Duistermaat - Van der Kallen in dimension one}

In this section we will discuss a generalisation of Duistermaat - Van der Kallen. 
Note that
    \begin{align*} \varphi\colon \C[X,X^{-1}] \to \C, f \mapsto f_0\end{align*}
is a linear map of $\C$-vector spaces for which $\varphi(X^N) = 0$ for all $|N|\geq 1$. 
The theorem of Duistermaat and Van der Kallen can thus be phrased as: $\Ker \varphi$ is an MZ-space.

We now investigate arbitrary $\C$-linear maps of a form, $L\colon \C[X,X^{-1}] \to \C$ or $L \colon \C[X] \to \C$.

\begin{remark} \begin{enumerate} 
\item If $L$ is injective, i.e., $\Ker L = 0$, then $\Ker L$ is an MZ-space.
\item If $L = 0$, i.e., the trivial linear map, then $\Ker L = \C[X,X^{-1}]$ (or indeed $\Ker L = \C[X]$), an MZ-space.
\end{enumerate}
\end{remark}

From now on, when we write $L$ is a linear map, we mean a non-trivial non-injective linear map. 

\begin{lemma} \label{lemma1} \label{enigma} 
Let $L \colon \C[X] \to \C$ be a $\C$-linear map for which there exists an $N\geq 1$ such that $L(X^n) = 0$ for all $n\in \Z_{\geq N}$. 
Then $\Ker L$ is an MZ-space of $\C[X]$ if and only if $L(1) \neq 0$.
\end{lemma}

\begin{lproof} $\Rightarrow:)$ Suppose that $L(1)=0$, then $1\in \Ker L$. 
Since $\Ker L$ is an MZ-space of $\C[X]$, we find that $\Ker L = \C[X]$, i.e., $L=0$, a contradiction. 
So $L(1)\neq 0$.

$\Leftarrow:)$ Note that $L(X^N) = 0$. 
So for any $f= \sum_{i=0}^n a_iX^i$ we have
    \begin{align*} L(fX^N) = L\left(\sum_{i=0}^n a_i X^{i+N}\right) = \sum_{i=0}^n a_i L(X^{i+N}) = 0,\end{align*}
hence $(X^N) \subset \Ker L$. 
Then by Corollary \ref{cor224} it suffices to show that $\Ker L/(X^N)$ is an MZ-space of $\C[X]/(X^N)$. 
We will also use Zhao's Theorem. 
We therefore need to show that all elements in $r(\Ker L/ (X^N))$ are algebraic over $\C$ and that $\C[X]/(X^N)e \subset \Ker L/(X^N)$ for all idempotents $e$ in $\Ker L/(X^N)$. 

Note first that $\Ker L/(X^N) \subset \C[X]/(X^N)$, hence any idempotent of $\Ker L/(X^N)$ is an idempotent of $\C[X]/(X^N)$. 
We start by determining the idempotents of $\C[X]/(X^N)$. 
Since $(X) \subset \C[X]$ is a maximal ideal, we find that $\C[X]/(X^N)$ is a local ring. 
A local ring has only idempotents $0,1$. 
Since $L(1)\neq 0$, we see that $0$ is the only idempotent in $\Ker L/(X^N)$, and we trivially have $0 \subset \Ker L/(X^N)$.

Hence it remains to show that every element in $r(\Ker L/(X^N))$ is algebraic over $\C$. 
Note that \begin{align*} r(\Ker L/(X^N)) \subset \C[X]/(X^N),\end{align*} and every element of $\C[X]/(X^N)$ is algebraic over $\C$ due to its dimension being finite.

Since we now satisfy the hypotheses of Zhao's Theorem, we find that $\Ker L/(X^N)$ is an MZ-space of $\C[X]/(X^N)$ and hence $\Ker L$ is an MZ-space of $\C[X]$. \end{lproof}

Note that the above lemma can be generalized to an arbitrary field $k$, with an identical proof. 
The following is the generalization of the $1$-dimensional version of the theorem of Duistermaat and Van der Kallen.

\begin{theorem} \label{thm136} \label{moregeneral} 
Let $L \colon \C[X,X^{-1}] \to \C$ be a $\C$-linear map for which there exists an $N\geq 1$ such that $L(X^n) = 0$ for all $n\in \Z_{\geq N}$ and all $n\in \Z_{\leq - N}$. 
Then $\Ker L$ is an MZ-space of $\C[X,X^{-1}]$ if and only if $L(1)\neq 0$.
\end{theorem}

The proof of this theorem relies on the following lemmas and is based on the proof of the $1$-dimensional Duistermaat-Van der Kallen theorem by Monsky \cite{Monskyproof}, \cite{monskylemma} (Lemmas 3.6 and 3.7).

\begin{lemma} \label{yearofthecat} 
Let $L \colon \C[X,X^{-1}] \to \C$ be a $\C$-linear map for which there exists an $N\geq 1$ such that $L(X^n) = 0$ for all $n\in \Z_{\geq N}$ and all $n\in \Z_{\leq - N}$. 
If $L(1)\neq 0$ and $f\in r'(\Ker L)$, then $f\in \C[X]$ or $f\in \C[X^{-1}]$.
\end{lemma}

\begin{lproof} 
Suppose $f\not\in \C[X]$ and $f\not\in \C[X^{-1}]$, i.e. $f = \alpha X^{-s} + \ldots + \beta X^{r}$ with $\alpha,\beta\neq 0$ and $s,r\geq 1$. 
Then consider the power series
    \begin{align*} W(z) = \sum_{j=0}^{\infty} L(f^j) z^j.\end{align*}
We will show that $W(z) \neq 1$. 
That is, $L(f^m) \neq 0$ for some $m\geq 1$. 
This contradicts $f\in r'(\Ker L)$.

Write $U(X) = X^s(1-zf(X)) \in \C(z)[X] \subset \C((z))[X]$ and $n:= r+s$. 
Then by the Newton-Puiseux Theorem (see Theorem \ref{newtonpuiseux}) there exists some $p\geq 1$ such that
    \begin{align*} U(X) = (-\beta z)(X-a_1)\cdots(X-a_n),\ \text{with all } a_i\in \C((z^{\frac1p})).\end{align*}
Write $k := \C((z^{\frac1p}))$. 
There we have the valuation $\nu$ defined by $\nu(z^{\frac1p}) = \frac1p$ and $k$ is complete with respect to this valuation. Furthermore $k[[X,X^{-1}]]$ is a ring, by Theorem \ref{spandauballet}. 
Extend $L$ to $k[[X,X^{-1}]]$ as
    \begin{align*} L \colon k[[X,X^{-1}]] \to k\end{align*}
by 
\begin{align*}L\left(\sum_{n=-\infty}^{\infty} c_nX^n\right) = \sum_{n=-(N-1)}^{N-1} c_n L(X^n).\end{align*}

Since $\C[X,X^{-1}][[z]]$ is a subring of $k[[X,X^{-1}]]$ (see Proposition \ref{true}), we have
    \begin{align*} \sum_{m\geq 0} f^m z^m \in k[[X,X^{-1}]].\end{align*}
Then $1-zf$ is invertible in $k[[X,X^{-1}]]$ and so is $U(X)$. 
Write
    \begin{align*} W(z) = L\left(\frac{1}{1-zf}\right) = L\left(\frac{X^s}{U(X)}\right).\end{align*}
We will now use the partial fractions decomposition of $X^s/U(X)$. 
Therefore, note that since we have $(-1)^n a_1\cdots a_n = \alpha/\beta$, we find that $a_i\neq 0$ for all $i=1,\ldots,n$. 
From $U(a_i)=0$ we find then that $f(a_i)=1/z$ for all $i=1,\ldots,n$.

So $f'(a_i)a_i' = -z^{-2}$. 
Also $U'(X) = sX^{s-1}(1-zf(X)) + X^s(-z)f'(X)$ and $U'(a_i) = -za_i^sf'(a_i) \neq 0$. 
Hence all $a_i$ are distinct. 
Then by Theorem \ref{gold} we get a partial fractions decomposition of the form
    \begin{align*} \frac{X^s}{U(X)} = \sum_{i=1}^n \frac{A_i}{X-a_i}\end{align*}
with $A_i = \frac{a_i^s}{U'(a_i)} = -\frac1{zf'(a_i)}$. 
Therefore
    \begin{align} \frac1{1-zf(X)} = \frac{X^s}{U(X)} = \sum_{i=1}^n - \frac1{zf'(a_i)(X-a_i)}.\end{align}
We now compute the inverse of each factor $X-a_i$ in $k[[X,X^{-1}]]$. 
Observe that $f(a_i) = 1/z$ implies that $\nu(a_i)\neq 0$. Then $\nu(a_i)>0$ or $\nu(a_i)<0$.

If $\nu(a_i)>0$, then we have
    \begin{align*} (X-a_i)^{-1} = X^{-1}(1-a_iX^{-1})^{-1} = X^{-1} \sum_{m=0}^{\infty} a_i^m X^{-m} \in k[[X,X^{-1}]]\end{align*}
while if $\nu(a_i)<0$, then
    \begin{align*} (X-a_i)^{-1} = -a_i^{-1}(1-a_i^{-1}X)^{-1} = -a_i^{-1} \sum_{m=0}^{\infty} (a_i^{-1}X)^m \in k[[X,X^{-1}]].\end{align*}

So we find that 
    \begin{align*}
        \frac1{1-zf(X)} & =  - \left( \sum_{i\in S^+} \frac1{zf'(a_i)(X-a_i)} + \sum_{i\in S^{-}} \frac1{zf'(a_i)(X-a_i)} \right) \\
        & =  - \left( \sum_{i\in S^+} \left( \frac1{zf'(a_i)}\right)\left( \sum_{m=0}^{\infty} a_i^m X^{-m}\right) X^{-1} - \sum_{i\in S^-} \left( \frac{1}{zf'(a_i)}\right) \left(\sum_{m=0}^{\infty} (a_i^{-1}X)^m \right) a_i^{-1} \right)
    \end{align*}
where $S^+ = \{ i \mid v(a_i) > 0\}$ en $S^- = \{i \mid v(a_i) <0 \}$.

Then \begin{align*}
        W(z) & =  L\left(\frac1{1-zf(X)}\right) \\
            & =  L\left( -\sum_{i\in S^+} \left( \frac1{zf'(a_i)}\right)\left( \sum_{m=0}^{\infty} a_i^m X^{-m}\right) X^{-1}\right) + L\left(\sum_{i\in S^-} \left( \frac{1}{zf'(a_i)}\right) \left(\sum_{m=0}^{\infty} (a_i^{-1}X)^m \right) a_i^{-1} \right) \\
            & =  -\sum_{i\in S^+}  \frac1{zf'(a_i)}  L\left(X^{-1}\sum_{m=0}^{\infty} a_i^mX^{-m}\right) + \sum_{i\in S^-} \frac{1}{zf'(a_i)} L\left(a_i^{-1} \sum_{m=0}^{\infty} (a_i^{-1}X)^m \right) \\
            & =  - \sum_{i\in S^+} \frac1{zf'(a_i)} \sum_{m=0}^{N-1} a_i^m L(X^{-(m+1)}) + \sum_{i\in S^-} \frac1{zf'(a_i)} \sum_{m=0}^{N-1} a_i^{-(m+1)} L(X^m)
    \end{align*}
Since $f'(a_i)a_i'=-z^{-2}$ we have 
    \begin{align*} W(z) = \sum_{i\in S^+} \sum_{m=0}^{N-1} z a_i' a_i^m L(X^{-(m+1)}) - \sum_{i\in S^-} \sum_{m=0}^{N-1} za_i'a_i^{-(m+1)} L(X^m)\end{align*}
We want to show $W(z)\neq 1$, i.e. we need to show that 
    \begin{align} \label{temporaryequation} \sum_{i\in S^+} \sum_{m=0}^{N-1}  a_i' a_i^m L(X^{-(m+1)}) - \sum_{i\in S^-} \sum_{m=0}^{N-1} a_i'a_i^{-(m+1)} L(X^m) \neq \frac1z \end{align}
To prove this inequality, we study the $a_i$ at infinity, i.e. we set $t= \frac1z$. 
Now fix an $i$. 
Then $f(a_i)=t$. 
Since $\C(z)=\C(t)$ we find that $a_i$ is algebraic over $\C(t)$ and hence also over $\C((t))$. 
Then again by the Newton-Puiseux theorem we can regard $a_i$ inside $\C((t^{\frac1p}))$ for some $p\geq 1$. 
Since $a_i\neq 0$ we can write $a_i = \sum_{n=m}^{\infty} c_nt^{n/p}$ for some $c_i \in \C$ with $c_m \neq 0$. 
Write $w$ for the valuation on $\C((t^{1/p}))$ defined by $w(t^{1/p}) = 1/p$. Then $w(a_i) = m/p$.

Suppose that $w(a_i)>0$, then 
    \begin{align*}
    w(f(a_i)) = w(\alpha c_m^{-s} t^{-ms/p}) = -ms/p = -s w(a_i) <0
    \end{align*}
 since $s\geq 1$. 
 But since $f(a_i)=t$ we also have $w(f(a_i))= w(t)=1$, a contradiction.

Similarly, if $w(a_i)<0$, then $w(f(a_i)) = w(\beta c_m^r t^{mr/p}) = mr/p = w(a_i)r < 0$, a contradiction.

Hence $w(a_i)=0$ and $a_i = \sum_{n=0}^{\infty} c_nt^{n/p}$ with $c_0\in \C^*$. 

Note that $t = \frac1z$, and $w(t)=1$. 
So if we show that $w(a_i'a_i^m)>1$ for all $m\in \Z$ we have shown our inequality \ref{temporaryequation}, and hence that $W(z)\neq 1$. 
Note that since $f(a_i)=t$ we have $a_i\not\in\C$. 
Hence there exists some $j>0$ with $c_j\neq 0$. 
Choose such $j$ minimal and write $a_i = c_0 + c_j t^{j/p} + R$ with $w(R) > j/p$. 
Then $w\left(\frac{da_i}{dt}\right) = \frac j p -1$.

Also $\frac{a_i'(z)}{a_i(z)} = -\frac{\frac{da_i}{dt}}{a_i}t^2$ and $w(a_i) = 0$, therefore we have
    \begin{align*} w(a_i'a_i^m) = w(-\frac{d a_i}{dt} t^2a_i^m) = \frac{j}p - 1 + 2 > 1,\end{align*}
which concludes the proof. 
\end{lproof}

\begin{lemma} \label{alstewart} 
Let $L \colon \C[X,X^{-1}] \to \C$ be a $\C$-linear map for which there exists an $N\geq 1$ such that $L(X^n) = 0$ for all $n\in \Z_{\geq N}$ and all $n\in \Z_{\leq - N}$. 
If $L(1) \neq 0$ and $f\in \C[X] \cap r(\Ker L)$, then $f\in X\C[X]$.
\end{lemma}

\begin{lproof} 
Let $f\in \C[X] \cap r(\Ker L)$, then $f\in r(\Ker L_{\mid \C[X]})$. 
By Corollary \ref{suchashame} and Lemma \ref{enigma} we know that $r(\Ker L_{\mid \C[X]})$ is an ideal of $\C[X]$. 
Since $L(1)\neq 0$ we find that $\Ker L_{\mid \C[X]} \neq \C[X]$, and hence $r(\Ker L_{\mid \C[X]}) \neq \C[X]$. 
As $r(\Ker L_{\mid \C[X]})$ is a principal ideal ($\C[X]$ is Euclidean), it is generated by some monic $g \in \C[X]\setminus \C$.
Since it is clear that $X\in r(\Ker L_{\mid \C[X]})$ we find $(X) \subset r(\Ker L_{\mid \C[X]})$.  
Hence $X$ is a multiple of this $g$. 
Hence $g=X$. 
So $r(\Ker L_{\mid \C[X]}) = (X)$ and $f\in (X)=X\C[X]$. 
\end{lproof}

We will use the following definition(s):

\begin{definition} 
For $f\in \C[X,X^{-1}]\setminus\{0\}$ we have the following quantities:
    \begin{align*} \deg_+ f = \max\{ i : f_i \neq 0 \}\end{align*}
and
    \begin{align*} \deg_- f = \min\{ i : f_i \neq 0 \}.\end{align*}
\end{definition}

\begin{proof} \emph{(of Theorem \ref{thm136})} $\Rightarrow:)$ Identical to the proof of Lemma \ref{lemma1}.

$\Leftarrow:)$ Suppose that $L(1)\neq 0$ and $f\in r(\Ker L)$. 
We want to show that $gf^m \in \Ker L$ for all $m \gg 0$. 
By Lemma \ref{yearofthecat} we find that $f^N\in \C[X]$ or $f^N\in \C[X^{-1}]$ for some $N\geq 1$, but then also $f\in \C[X]$ or $f\in \C[X]$.

Assume that $f\in \C[X]$, the argument is similar when $f\in \C[X^{-1}]$. 
Then $f\in \C[X] \cap r(\Ker L)$. By Lemma \ref{alstewart} we then find $f\in X\C[X]$. 
Hence $L(f^m) = 0$ for all $m\geq N$. 

Then $L(gf^m) = 0$ for all $m \geq N - \deg_- g$. \qed
\end{proof}

Theorem \ref{thm136} is false in characteristic $p>0$, as was shown in \cite{Wil}.

\section{Encore: Generalizations}

In this section we list several generalizations of various results we have seen in this paper. 
The proofs can be found in \cite{Thesis}.

\begin{proposition} \emph{(Generalization of Lemma \ref{lemma1})} 
Let $L\colon \C[X_1,\ldots,X_n] \to \C$ be a $\C$-linear map for which there exists an $N\geq 1$ such that $(X_1,\ldots,X_n)^l \subset \Ker L$ for all $l \geq N$. 
Then $\Ker L$ is an MZ-space of $\C[X_1,\ldots,X_n]$ if and only if $L(1) \neq 0$. 
\end{proposition}

We cannot prove
    \begin{indentation}{10mm}{0mm}
     Let $L \colon \C[X_1,\ldots,X_n] \to \C$ be a $\C$-linear map for which there exists an $N\geq 1$ such that $L(X_i^m) = 0$ for all $m\in \Z_{\geq N}$ and all $i \in \{1,\ldots,n\}$. 
     Then $\Ker L$ is an MZ-space of $\C[X_1,\ldots,X_n]$ if and only if $L(1)\neq 0$.
    \end{indentation}
For this we have the following counterexample:

\begin{example} 
Take $n=2$ and $N=2$ in the following form. 
Let $L \colon \C[X_1,X_2] \to \C$ be the $\C$-linear map given by \begin{align*} L(X_1^iX_2^j) =  \begin{cases}
 1 &\mbox{if } i=1 \mbox{ or } i=j=0 \\
0 & \mbox{else } \end{cases} \end{align*}
Then we have $L(1)=1$, and $L(X_i^m) = 0$ for all $m \geq 2$. 
It is then clear that $X_2 \in r(\Ker L)$ by construction, but $X_1X_2^m \not\in \Ker L$ for all $m\geq 1$. 
Therefore $\Ker L$ is not an MZ-space. 
\end{example}

While it is technically not a generalization, we also have the following lemma, which uses the same method-of-proof:

\begin{proposition} 
Let $k$ be a field, $A$ a \emph{finite-dimensional} $k$-algebra and $\m m$ a maximal ideal of $A$. 
Let $L \colon A \to k$ be a $k$-linear map for which there exists an $N\geq 1$ such that $\m m^l \subset \Ker L$ for all $l\geq N$. 
Then $\Ker L$ is an MZ-space of $A$ if and only if $L(1)\neq 0$.
\end{proposition}

In this section we discuss a generalization of Zhao's Idempotency Theorem (Theorem \ref{zhao}), that is in the exercises of \cite{Ess2}.

\begin{definition} 
Let $R$ be a ring, $S$ a multiplicatively closed subset of $R$ and $A$ an $R$-algebra. 
An $R$-submodule $M$ of $A$ is called \emph{$S$-saturated} if we have $a\in M$ for all $a\in A$ for which there exists some $s\in S$ such that $sa\in M$, i.e., 
    \begin{align*} \{ a\in A \mid \exists s\in S[sa\in M]\} \subset M.\end{align*}
\end{definition}

\begin{definition} 
Let $R$ be a ring, $S$ a multiplicatively closed subset of $R$ and $A$ an $R$-algebra. 
Then $e\in A$ is an \emph{$S$-idempotent} if there exists some $s\in S\setminus\{0\}$ such that $e^2=se$.
\end{definition}

\begin{theorem}[Generalized Zhao Theorem \cite{Ess2}]
Let $R$ be a domain, $S= R\setminus\{0\}$ and $A$ an $R$-algebra. 
Let $M$ be an $S$-saturated $R$-submodule of $A$ such that all elements of $r(M)$ are algebraic over $R$. 
Then $M$ is an MZ-space of $A$ if and only if for every $S$-idempotent $e$ of $A$ which belongs to $M$ we have that for every $a\in A$ there exists an $s\in S$ such that $sae\in M$.
\end{theorem}

Using this Generalized Zhao Theorem, we can prove the following generalization of Lemma \ref{lemma1}:

\begin{proposition} 
Let $R$ be a domain and $L \colon R[Y] \to R$ be a linear map such that there exists an $N\geq 1$ for which $L(Y^n) = 0$ for all $n\geq N$. 
Then $\Ker L$ is an MZ-space of $R[Y]$ if and only if $L(1)\neq 0$.
\end{proposition}

 \bibliographystyle{amsplain}
 \bibliography{refs}

\end{document}